\documentclass[leqno]{article}

\usepackage{amsmath,amssymb,amsthm}
\usepackage{mathrsfs}
\usepackage{tikz}
\usetikzlibrary{cd}
\usepackage{enumitem}
\usepackage[hidelinks]{hyperref}
\usepackage{tabularx}

\newcommand{\A}{\mathbb{A}}
\newcommand{\BGtorus}[1]{(\A^N \setminus \setl{0})^{#1}}

\newcommand{\C}{\mathbb{C}}
\newcommand{\cadivQ}{\cadiv_{\Q}}
\newcommand{\Deltap}{\Delta_P}
\newcommand{\dfS}{\mathcal{S}}
\newcommand{\dual}[1]{{#1}^{\vee}}
\newcommand{\gitquot}{\operatorname{/\!\!/}}
\newcommand{\glsec}[2]{\Gamma\left( #1,\,#2 \right)}

\newcommand{\inpr}[2]{\left\langle #1,\, #2 \right\rangle}

\newcommand{\locD}{\text{Loc}(D)}
\newcommand{\N}{\mathbb{N}}
\newcommand{\orth}[1]{#1^{\perp}}
\newcommand{\plD}{\mathfrak{D}}
\newcommand{\Pol}[2]{\Polsymb(#1, #2)}
\newcommand{\PolNs}{\Pol{N}{\sigma}}
\newcommand{\Q}{\mathbb{Q}}
\newcommand{\relSpec}{\textbf{\textup{Spec}}}
\newcommand{\restrict}[2]{\left.#1\right|_{#2}}
\newcommand{\set}[2]{\left\{#1\,\middle|\,#2\right\}}
\newcommand{\setl}[1]{\left\{#1\right\}}
\newcommand{\shf}[1]{\mathcal{A}}
\newcommand{\shfA}{\shf{A}}
\newcommand{\strshf}[1]{\mathscr{O}_{#1}}

\newcommand{\tupI}{\mathscr{T}_I}
\newcommand{\tvar}[1][T]{\ensuremath{#1}-\text{variety}}
\newcommand{\tvars}[1][T]{\ensuremath{#1}-\text{varieties}}

\newcommand{\Z}{\mathbb{Z}}

\newcommand{\Nprime}{N^\prime}

\newcommand{\NQ}{N_\Q}
\newcommand{\MQ}{M_\Q}
\newcommand{\Xip}{\Xi_P}
\newcommand{\Pl}{\mathbb{P}}

\DeclareMathOperator{\cadiv}{CaDiv}

\DeclareMathOperator{\Polsymb}{Pol}

\DeclareMathOperator{\relint}{rel\,int}
\DeclareMathOperator{\SL}{SL}
\DeclareMathOperator{\Spec}{Spec}
\DeclareMathOperator{\tail}{tail}
\DeclareMathOperator{\WDiv}{WDiv}

\numberwithin{equation}{subsection}
\theoremstyle{plain}
\newtheorem{theorem}[equation]{Theorem}
\newtheorem{lemma}[equation]{Lemma}

\newtheorem{proposition}[equation]{Proposition}
\theoremstyle{definition}
\newtheorem{definition}[equation]{Definition}
\newtheorem{notation}[equation]{Notation}
\theoremstyle{remark}
\newtheorem{remark}[equation]{Remark}
\newtheorem{example}[equation]{Example}

\title{
  Equivariant Chow groups of a complexity one \tvar{}
}
\author{
  Pavankumar Dighe\footnote{\textit{Email:} \texttt{pavankumardighe@gmail.com}}
  \and
  Vivek Mohan Mallick\footnote{\textit{Email:} \texttt{vmallick@iiserpune.ac.in}}
}

\begin{document}
\maketitle
\begin{abstract}
This paper provides a computation of the equivariant Chow group of a rational, complete, complexity one $T$-variety.
\end{abstract}

\section{Introduction}

A \tvar{} is an algebraic variety along with an effective action of an algebraic torus $T$.
In this paper, we study the $T$-equivariant Chow groups of such varieties.

The study of equivariant invariants of varieties have a rich history;
see, for example, Timashev \cite{timashev:gmanifoldscomp1,timashev:comp1torusact}.
In this paper, we are interested in \tvars{}, and our approach
follows that of Altmann, Hausen and Suss \cite{ah:affinetvar, ahs:gentvar}.
In this approach, the \tvars{} are described in terms of pp-divisors, which are 
generalizations of usual divisors, whose coefficients are polyhedra satisfying some properties.
Quite a lot is known about these spaces, specially if the \tvar{} is of complexity one,
i.e.\ when the dimension of the torus is one less than the dimension of the \tvar{}.
The article \cite{aipsv:geomtvar} provides an comprehensive summary of what is
known about the geometry of \tvars{} including divisors,
cohomology of line bundles, intersection theory, etc.
Their topology, including a computation of Hodge-Deligne numbers, cohomology ring and
fundametal groups were studied in \cite{lafaceliendomoraga:topologyTvar}.
Vector bundles over \tvars{} with compatible torus actions have also been studied by
Ilten and Su{\ss} \cite{iltensuss:eqvectbundle}.
The algebro-combinatorial data describing \tvars{} make them a natural source of examples
where one can hope to compute various properties.

A natural question is to ask if one can compute various invariants for these varieties.
One such important invariant is the Chow group which also acts as a natural habitat for
characteristic classes like Chern classes to live in.
This is an old question (see, for example, the work on chow groups of $\SL(2)$ embeddings
by Moser-Jauslin \cite{moserjauslin:chowringsl2},
and Gonzales \cite{gonzales:eqopchowring}.
Under certain conditions, N{\o}dland computed Chow groups of \tvars{} \cite{Ndland2021}.
It was further studied by Botero \cite{botero:chowringTvar}.

Since we are in an equivariant setting, a natural question is to ask if we can compute
equivariant versions of these invariants.
Equivariant Chow groups, modelled after the equivariant cohomology of Borel, 
were defined by Edidin and Graham \cite{eg:equivintth} extending previous work done by
Briney, Gillet and Vistoli.
We shall describe a way to compute these groups.

The paper is organized as follows.
Section \ref{sec:prlmtveq} is a summary of notions we need for this paper.
The theory of \tvars{} is reviewed in the subsection \ref{ssc:tvarieties}, followed by
a quick recall of equivariant Chow groups in subsection \ref{ssc:eqchowgp}.
Section \ref{sec:esteqchg} gives the main results leading to the computation of the
equivariant Chow groups using Ilten and Vollmert's technique of 
downgrading \cite{iltenvollmert:upgrading}.
This reduces the problem to computation of usual Chow groups of certain \tvars{}.
This is done in section \ref{sec:teqchtvr} following N{\o}dland \cite{Ndland2021}.
In section \ref{sec:rkvktkcn}, we an analogue of a result of N{\o}dland 
\cite[Proposition 6.3]{Ndland2021} for equivariant Chow groups (see \ref{pro:rvksumcn}).
We end the paper by computing the equivariant versions numbers involved for an example
already considered by N{\o}dland in the non-equivariant setting.

\subsection*{Acknowlegements}
The first author thanks UGC (UGC-Ref.No: 1213/(CSIR-UGC NET JUNE 2017) ) for funding the research
and IISER Pune for providing facilities. The second author thanks IISER Pune for providing 
an excellent environment to conduct this research.

\section{Some preliminaries} \label{sec:prlmtveq}
In this section, we review some facts about $T$-varieties and equivariant Chow groups.
\subsection{\tvars{}}
\label{ssc:tvarieties}

$T$-varieties are normal algebraic varieties with an effective torus action.
 These varieties can be described by a combination of algebraic geometry and  combinatorial data. A \emph{tail cone}, $\tail(\Delta)$, of any polyhedron $\Delta \subset
\NQ$ is defined as
\begin{equation*}
  \tail(\Delta) := \set{v \in \NQ}{v + \Delta \subseteq \Delta}.
\end{equation*}
Fix a strongly convex integral polyhedral cone $\sigma$ in $\NQ$.
The polyhedra with tail cone $\sigma$ form a semigroup under Minkowski sum,
with $\sigma$ as its additive identity. We denote a semigroup by $\text{Pol}^+(N,\sigma)$.
Let $\PolNs$ be the associated Grothendieck group.
A \emph{polyhedral divisor} on a normal variety $Y$ with a tail cone $\sigma$ is a formal sum
$\plD = \sum_i \Delta_i \otimes D_i$ where $D_i$ runs over all prime
divisors of $Y$ and the coefficients $\Delta_i \in \PolNs$; that is $\plD
\in \PolNs \otimes \WDiv(Y)$, where $\WDiv(Y)$ is the group of Weil divisors
on $Y$.

Each polyhedral devisor defines an \emph{evaluation map}  
\begin{equation} \label{domainoflin}
  \plD \colon \sigma^\vee \to \WDiv_\Q (Y)
\end{equation}
where $\WDiv_\Q(Y) = \WDiv(Y) \otimes_{\Z} \Q$.
This is defined by 
\begin{equation*}
  \plD(u) := \sum_i \biggl( \min_{v \in \Delta_i} \inpr{u}{v} \biggr) D_i.
\end{equation*}
The evaluation map, $\plD$, is a piecewise linear convex function.
In fact, there is a bijection between polyhedral divisors with a fixed tail
cone $\sigma$ and convex piecewise linear maps  $ \sigma^\vee \to
\WDiv_\Q(Y)$ (cf \cite[proposition 1.5]{ah:affinetvar}).

\begin{definition}
A  polyhedral divisor $\plD$ is \emph{semiample} if and only if $\plD(u)$ is
  semiample for all $ u \in \sigma^\vee$.
\end{definition}

\begin{definition}
  Let $Y$ be a normal variety and let $\sigma \subset \NQ$ be a pointed
  cone.
  A \emph{proper polyhedral divisor} or a \emph{pp-divisor} on $(Y, N)$ with respect
  to $\sigma$ is a semiample polyhedral divisor $\plD = \sum_i \Delta_i
  \otimes D_i $ such that $D_i$ is an effective Weil divisor, $\Delta_i \in \text{Pol}^+(N,\sigma)$ with only finitely many are different from $\sigma$, and for every
  $u \in \sigma^\vee,$ $\plD(u)$ is a rational Cartier divisor on $Y$, which
  is big whenever $u$ belongs to the relative interior of $\sigma^{\vee}$.
\end{definition}

\begin{definition}
  \label{def:afftvarpld}
  To a pp-divisor $\plD$, one associates an affine scheme as follows.
  Let, for $u \in \dual{\sigma} \cap M$, $\shfA_u = \strshf{Y}(\plD(u))$.
  Then, $\shfA = \bigoplus_{u \in \dual{\sigma} \cap M} \shfA_u$ is a sheaf
  of $\strshf{Y}$ algebras. Define, the \emph{affine scheme associated with}
  $\plD$ to be
  \begin{equation*}
    X(\plD) = \Spec \left( \bigoplus_{u \in \dual{\sigma} \cap M}
    \glsec{Y}{\shfA_u} \right) = \Spec \glsec{Y}{\shfA}.
  \end{equation*}
  
\end{definition}

\begin{theorem}{\cite[Theorems 3.1 and 3.4]{ah:affinetvar}}
  \label{thm:affinetvar}
  Let $Y$ be a normal, semiprojective variety.  Fix a lattice $N$, its dual
  $M$, and let $\sigma$ be a strongly convex polyhedral cone in $\NQ$. Given
  a pp-divisor $\plD$ on $(Y, N)$ consider the affine scheme $X = X(\plD)$
  associated to $\plD$. Also, define
  \begin{equation*}
    T = \Spec \C[M] \qquad \text{ and } \qquad \tilde{X} = \tilde{X}(\plD) =
    \relSpec_Y \shfA 
  \end{equation*}
  where $\relSpec_Y \shfA$ is the relative spectrum of the sheaf of algebras
  $\shfA$ over $Y$. Then (See \cite[theorems 3.1 and 3.4]{ah:affinetvar})
  both $X$ and $\tilde{X}$ are normal varieties of dimension $\dim Y + \dim
  T$ and admit an effective $T$-action. There are two canonical maps: $\pi \colon
  \tilde{X} \longrightarrow Y$ which is a good quotient and a proper
  birational $T$-equivariant contraction morphism $r \colon \tilde{X}
  \longrightarrow X$. The variety $Y$ is the Chow quotient (cf \cite{MR1119943}) of $X$ under the
  action of $T$.
  \[
    \begin{tikzcd}
      \tilde{X} \arrow[r, "r"] \arrow[dr,"\pi"']  &
      X  \\
      & Y.
    \end{tikzcd}
  \]
  The category of affine $T$-varieties and dominant $T$-equivariant morphisms is
  equivalent to the category of pp-divisors along with properly defined
  morphisms (cf \cite[theorem 8.8 and corollary 8.14]{ah:affinetvar}).
\end{theorem}

From \cite{ahs:gentvar}, we are going to recall gluing of affine $T$-varieties of complexiy one . To describe combinatorial data for a non-affine $T$-variety we allow empty coefficients in $\plD$. For a polyhedral divisor $\plD=\sum \Delta_D\otimes D$ on a curve $Y$ where $D$ denote  Weil divisor on curve $Y$, where some $\Delta_D$'s can be $\phi$. Such a $\plD$ is said to be a pp-divisor on $Y$ if $\plD|_{\locD}$ is pp-divisor on $\text{Loc}(\plD)$ where $\text{Loc}(\plD)= Y\setminus\bigcup_{\Delta_D = \phi}D$.

\begin{definition} Let $\plD=\sum_p \Deltap \otimes P$, $\plD^\prime = \sum_P \Deltap^\prime \otimes P$ be two polyhedral divisor on $Y$ with tail cone $\sigma$ and $\sigma^\prime$.
    \begin{itemize}
        \item Intersection of polyhedral divisors is defined by
        \[\plD \cap \plD^\prime = \sum_P (\Deltap \cap \Deltap^\prime) \otimes P.\]
        \item For $p \in Y $ slice associated to $p$ is defined by $\plD_p=\underset{P  :  p \in P}{\sum} \Deltap$.
        \item We say $\plD^\prime \subset \plD$ if $\Deltap^\prime \subset \Deltap$ for every $P \in Y$
   
    \item For a proper polyhedral divisor $\plD^\prime \subset \plD$  we have a dominant morphism $X(\plD^\prime) \to X(\plD)$. If this map is an open embedding, we call  $\plD^\prime$ is a \emph{face} of $\plD$ and it is denoted by $\plD^\prime \prec \plD$.
 \end{itemize}
\end{definition}

\begin{definition}
    A \emph{divisorial fan} is a finite set $\dfS$ of proper polyhedral divisors on $Y$ such that for $\plD^\prime, \plD \in\dfS$, we have $ \plD \cap \plD^\prime  \in \dfS $ and $\plD^\prime \succ \plD \cap \plD^\prime \prec \plD$. 
\end{definition}
\begin{definition}
    A proper polyhedral divisor $\dfS$ is \emph{contraction free} if $\text{Loc}(\plD)$ is affine and a divisorial fan $\dfS$ is contraction free if for all $\plD \in \dfS$, $\text{Loc}(\plD)$ is affine. This is equivalent to the statement that the contraction map $r$ is an isomorphism.  
\end{definition}

\begin{definition}
  \label{def:ppdivsor}
  Suppose a \tvar{} is described by $(Y, \dfS)$. We shall call $Y$ the
  \emph{the space of the pp-divisor}.
\end{definition}

Since we are interested in complexity one $T$-variety, we shall recall some notions from  \cite{iltenpolarized}. For a complete complexity one $T$-variety, the combinatorial description is completely determined by the slices. 
 \begin{definition}
     A \emph{marked fansy divisor} on a curve $Y$ is a formal sum $\Xi = \sum \Xip.P$
     with a fan $\Sigma$ and some subset $C \subset \Sigma$, such that 
     \begin{enumerate}
         \item $\Xip$ is complete polyhedral subdivision of $\NQ$ and tail($\Xip)=\Sigma$ for all $P \in Y$.
         \item For full-dimensional $\sigma \in C$ the polyhedral divisor 
         $D^\sigma =\sum \Deltap^\sigma \otimes P$ is proper, where $\Deltap^\sigma$ is a unique element of $\Xip$ with tail cone $\sigma$.
         \item For $\sigma \in C$ of a full dimension and $\tau \prec \sigma$ we have $\tau \in \Sigma$ if and only if deg$D^\sigma \cap \tau \neq \phi$.
          \end{enumerate}
          A collection of cones $C$ is called \emph{marked} if $\tau \prec \sigma$ and $\tau \in C$ implies $\sigma \in C$. Elements of $C$ are also called \emph{contracted cones}. For any complete divisorial fan one can associate a marked fansy divisor. For a $\dfS$ divisorial fan for each $P \in Y$, we have a slice $\dfS_P$. Then the corresponding marked fansy divisor is $\Xi= \sum \dfS_P.P$. 
        
 \end{definition}
         \begin{remark}
             If $\dfS$ is contraction free, then $C$ is empty and, conversely, for a given marked fansy divisor $\Xi$ with empty $C$ the corresponding $\dfS$ is contraction free \cite[proposition 1.6]{iltenpolarized}. 
         \end{remark}
         \begin{remark}
             A complexity one $T$-variety $X(\dfS)$ is complete if and only if each of it slices $\dfS_p$ is a complete subdivision of $\NQ$ and $Y$ is complete. 
         \end{remark} \label{remark2.10}
     
   In the following paragraph, we briefly recall the construction of a divisorial fan from a divisorial polyhedron (\cite[section 4 ]{iltenvollmert:upgrading}).

 \begin{definition}{}
Let $Y$ be a smooth curve, a \emph{divisorial polyhedron} consists of a pair $(\text{L}, \square)$, where $\square$ is a polyhedron in $\MQ$, and $\text{L} $ is piecewise affine concave map from $\square$ to $\text{CaDiv}_\Q Y$ taking values in samiample divisors.
     \end{definition}
     Let's recall how to associate divisorial polyhedron to a complexity one $T$-variety. For any $u \in \square$,  $\text{Lin}_P:\text{tail}(\square) \to \Q$ is defined by
 \[ \text{Lin}_{P}(v)   =  \lim_{\lambda \to \infty} L_{P}(u +\lambda . v )/\lambda \]
and we set
\begin{align*}
    \square^{*}_P &= \{ v \in \NQ \mid (v,w) \geq \text{Lin}_P(w)\ \forall w \in \square\};
    &
    \text{and }\text{L}^{*}_P& : \square^{*}_P \to Q \\
    &
    &
    \text{L}^{*}_P&(v) = \underset{ u \in \square }{\text{min} }  \biggl( \langle u,v \rangle - \text{L}_P(u) \biggl)
\end{align*}

Then, for any $P$, $\Xi(\text{L}^{*}_P)$ is subdivision of $ \square^{*}_P$ consisting of pointed polyhedra. Now from the above construction, we can associate a divisorial fan.
Define a set $K=\{ P \in Y | \; \text{L}_P \not\equiv 0 \}$ and $E=\underset{P \in K}{\sum}P$ \label{definitionE}. Then the divisorial fan $\dfS$ is generated by \[C_{\text{L}}=\{ \Delta_P\otimes P + \phi\otimes(E-P)  | \; P \in K, \; \Delta_P \in \Xi(\text{L}^{*}_P) \}.\]
\begin{remark}
    From \cite[proposition 4.2]{iltenvollmert:upgrading} $\dfS$ is a contraction free divisorial fan.
\end{remark}

 We are interested in the product of two $T$-varieties. The combinatorial data associated to a product is given by following lemma.

\begin{lemma}{\cite[Propositon 5]{aipsv:geomtvar}}
  \label{pro:prodtvar}
   For two pp-divisors $\plD^\prime=\sum \Delta_i^\prime\otimes D_i^\prime$ on $(Y^\prime, \sigma^\prime \subset \Nprime)$ and $\plD =\sum \Delta_i\otimes D_i$ on $(Y,\sigma \subset N)$, define $\plD \times \plD^\prime = \sum (\Delta_i \times \sigma^\prime)\otimes (D_i \times Y^\prime) + \sum (\sigma \times  \Delta_i^\prime)\otimes (Y \times D_i^\prime)  $ as a pp-divisor on $(Y\times Y^\prime , N \oplus \Nprime)$. Then $X(\plD \times \plD^\prime)=X(\plD) \times X(\plD^\prime)$.
  
\end{lemma}

\subsection{Equivariant Chow groups}
\label{ssc:eqchowgp}

For schemes with an action of a group $G$, Totaro \cite{totaro:Chowringclsp} and Edidin and Graham \cite{eg:equivintth} define an equivariant Chow group, which we recall now.
For an arbitrary group $G$, we have the construction of a $\Delta$-complex $EG$ on which $G$ acts freely on the left multiplication. We denote $BG=EG/G$.The spaces $EG$ and $BG$ are usually infinite dimensional spaces. For $G=\C^*$, we have  $EG=\C^\infty \setminus \setl{0}$ on which $G$ acts freely and $BG=\mathbb{P}^\infty$. But these spaces are inexplicable in algebraic geometry. Nonetheless, we have algebraic varieties $E_m = \C^m \setminus \setl{0}$ and $B_m=\mathbb{P}^{m-1}$ which provide approximations to 
$\C^\infty \setminus \setl{0} \to \mathbb{P}^{m-1}$. For two groups $G$ and $H$, $E(G \times H) \equiv  EG \times EH$. For an algebraic torus of dimension $d$ we have $ET={(\A^\infty \setminus \setl{0})^{d}}$ and $E^N_T = \BGtorus{d}$. Thus we denote quotient variety $(X \times \BGtorus{d})/T$ by $X_T$.
\begin{definition}[$k^{th}$ equivariant Chow group]
    For $X$, an $n+1$ dimensional $T$-variety of complexity one, and for $k \leq d+1$, the $k$-th  equivariant Chow group is defined as the usual $k^{th}$ Chow group of space $X_T$  i.e $A_k(X_T)$.
\end{definition}

\section{Estimating equivariant Chow groups of a \tvar{}}
\label{sec:esteqchg}

\subsection{A description of $E^N_T$}
\label{sse:descentt}.

In our case, we are lucky that the approximate space for a torus $T$ of dimension $d$, given by $\BGtorus{d}$ is a toric variety. The next lemma will give a description of this variety in terms of a fan. The cones of this fan are constructed as follows.

Note that for a $d$-dimensional torus $T$, a finite dimensional
approximation of the classifying space is given by
\begin{equation*}
  E^N_T = \biggl( \A^N \setminus \setl{0} \biggr)^d.
\end{equation*}
$E^N_T$ is a toric variety of dimension $Nd$, whose dense torus will be
denoted by $T_E$. The fan describing $E_T^N$ as a toric variety can be
constructed as follows.

Consider the $\Q$-vector space $\Q^N$ with the standard basis $e_1, \dotsc,
e_N$. Let $\theta$ be the cone generated by $\setl{e_1, \dotsc, e_N}$.  Let
$\tupI$ be the set of tuples of numbers between $1$ and $N$:
\begin{equation}
  \label{equ:tupled1N}
  \tupI = \set{(i_1, \dotsc, i_d) \in \N^d}{1 \leq i_j \leq N, \text{ for
  each } 1 \leq j \leq d}.
\end{equation}
For each $i \in \setl{1, 2, \dotsc, N}$, define $\sigma^i$ to be the face
$\theta \cap \orth{\rho_i}$ where $\rho_i$ is the ray generated by $e_i$.
The cone $\sigma^i$ is generated by $e_1, \dotsc, \widehat{e_i}, \dotsc,
e_N$, all the basis vectors except $e_i$. For $I = \setl{i_1, \dotsc, i_r}
\subset \setl{1, \dotsc, N}$, let
\begin{equation}
  \label{equ:conessgE}
  \sigma_I = \sigma_{i_1, \dotsc, i_r}  = \delta_1 \times \dotsb \times
  \delta_d;
\end{equation}
where $\delta_i = \sigma^i$ if $i \in \setl{i_1, \dotsc, i_r}$; and $\delta_i
= 0$ otherwise. Let $\Sigma_E$ be the fan generated by these cones in
$\Q^{Nd}$. The following lemma is evident.

\begin{lemma}
  \label{lem:fanENTtv}
  With notation, as described above, the toric variety corresponding to
  $\Sigma_E$ is exactly $E_T^N = \bigl( \A^N \setminus \setl{0} \bigr)^d$.
\end{lemma}

\subsection{Product of a $T$-variety and $E_T^N$}

Suppose $Y$ is a curve and $\dfS$ be a divisorial fan on $Y$. Suppose $X =
X(\dfS)$ be the corresponding affine \tvar{} (of complexity $1$). We wish to
describe $X(\dfS) \times E_T^N$ as a \tvar{}. Observe that $E_T^N$ is a
toric variety under the action of the torus $T_E$. Being an approximation of
the classifying space, $E_T^N$ comes with a natural action of $T$. Thus we
get a diagonal action of $T$ on $X(\dfS) \times E_T^N$. We wish to compute
the geometric quotient
\begin{equation*}
  (X(\dfS) \times E_T^N) / T
\end{equation*}
Our idea is that the geometric quotient will be birational to the space of any
pp-divisor (see definition \ref{def:ppdivsor}) representing $X(\dfS) \times
E_T^N$ is considered as a \tvar{} under the action of $T$. We begin by describing
$X(\dfS) \times E_T^N$ as a complexity $1$ \tvar{} under the action of $T
\times T_E$.

\begin{proposition}
  \label{pro:XtEc1tvr}
  Suppose $Y$ is a curve, $T$ a torus and $E_T^N$ the corresponding
  approximate classifying space.
  \begin{enumerate}
    \item \label{ite:affdscpd}
      (Affine case) Suppose $\plD = \sum_{i=1}^{n} \Delta_i \otimes
      \setl{p_i}$ is a pp-divisor on $Y$ and $X(\plD)$ is the corresponding
      affine \tvar{}. For $I \in \tupI$ (see equation \eqref{equ:tupled1N})
      define
      \begin{equation*}
	\plD_I = \sum_{i = 1}^{n} \bigl(\Delta_i \times \sigma_I\bigr)
	\otimes \setl{p_i}
      \end{equation*}
      where $\sigma_I$ is the cone defined in equation \eqref{equ:conessgE}.
      Let $\dfS_{\plD}$ be the divsorial fan generated by $\set{\plD_I}{I
      \in \tupI}$. Then, $X(\plD) \times E_T^N$, considered as a complexity
      $1$ \tvar{} under the action of $T \times T_E$, is described by
      $\dfS_{\plD}$.
    \item \label{ite:gendscpd}
      (General case) For a \tvar{} $X = X(\dfS)$, described by a divisorial
      fan $\dfS$ over a curve $Y$, $X \times E_T^N$ is desribed by the
      divisorial fan generated by $\set{\dfS_{\plD}}{\plD \in \dfS}$ where
      for each $\plD \in \dfS$, $\dfS_{\plD}$ is defined as in the affine
      case (item \ref{ite:affdscpd}).
  \end{enumerate}
\end{proposition}
\begin{proof}
  This is an easy consequence of the description of product of \tvars{} as a
  \tvar{} (see, for example, \cite[section 2.4]{aipsv:geomtvar}) which in turn follows from K{\"u}nneth formula
  \cite[6.7.8]{grothendieck:ega3ii}.
\end{proof}

\subsection{Computing the downgrade}

Let $Y$ be a smooth curve, and $\dfS$ be a divisorial fan. Let $X = X(\dfS)$ be a $T$-variety of codimension $1$. Note that $T \times T_{E}$ acts effectively on $X(\dfS) \times E^N_T$. By proposition \ref{pro:prodtvar}, one can write down the following description of this $T$-variety. To compute the torus-equivariant Chow group, we study $\bigl(X(\dfS) \times E^N_T\bigr) \gitquot T$. We know that the base of the pp-divisor in the description of $X(\dfS) \times E^N_T$ as a $T$-variety under the action of $T$ is birational to this quotient i.e. $ (X(\dfS) \times E^N_T\bigr) \gitquot T \;\text{birational to} \; \widetilde{Y_C}$.

\[
  \begin{tikzcd}
    \widetilde{X \times E^N_T} \arrow[r, "r_{E}"] \arrow[d, "\pi_C"] &
    X \times E^N_T  \\
    \widetilde{Y_C}  \arrow[d, "\pi"] \\
    Y,
  \end{tikzcd}
  \]

We compute this using downgrading following \cite{iltenvollmert:upgrading}. For $X = X(\dfS)$ a $T$-variety with a $d$-dimensional torus acting on it as above, let $M$ be the lattice of characters  for $T$ and $N$ be the dual of $M$. Let $M_{E}$ and its dual $N_{E}$ correspond to the torus $T_{E}$ acting on $E^N_T$. We mention a few maps to make the description of the downgraded $T$-variety easier. Consider the short exact sequence
\[
  \begin{tikzcd}
    0 \arrow[r] &
    M_{E} \arrow[r, "\iota"] &
    M \oplus M_{E} \arrow[r, "\pi"] \arrow[l, "t", bend left, dashed] &
    M \arrow[r] &
    0
  \end{tikzcd}
\]
where $\iota$ and $\pi$ are described in terms of the map $I \colon M_{E} \longrightarrow M$ which is given by the matrix (with respect to the standard basis)
\[
   I = 
   \begin{pmatrix}
     I_1 & \cdots & I_d
   \end{pmatrix}
\]
where each $I_i$ for $i = 1, \dotsc, d$ is a $d \times N$ matrix given by
\[
  I_i = 
  \begin{pmatrix}
     e_i & \cdots & e_i
  \end{pmatrix},
  \qquad i \in \setl{1, \dotsc, d}
\]
where $e_i$ is the $i$-th vector in the standard basis of $\Z^d$ with $1$ at the $i$-th position and $0$ elsewhere. Now,
\[
  \iota(b) = (-I(b), b) \qquad \text{and} \qquad \pi(a, b) = a + I(b).
\]

Dualizing this, we also get a short exact sequence
\begin{equation} \label{equ:defalprh}
\begin{tikzcd}
  0 \arrow[r] &
  N \arrow[r, "\alpha"] &
  N \oplus N_{E} \arrow[r, "\rho"] &
  N_{E} \ar[r] &
  0
\end{tikzcd}
\end{equation}
where the maps are described in terms of $J \colon N \longrightarrow N_{E}$ which is defined as
\[
 J =
 \begin{pmatrix}
   J_1 \\ \vdots \\ J_d
 \end{pmatrix}.
\]
Here each $J_i$, $i \in \setl{1, \dotsc, d}$ is an $N \times d$ matrix of the form
\[
  J_i^T = 
  \begin{pmatrix}
    e_i & \cdots & e_i
  \end{pmatrix}
\]
where $e_i$ as above is the $i$-th element of the standard basis for $\Z^d$. With this notation, we have
\[
  \alpha(a) = (a, J(a)) \qquad \text{and} \qquad \rho(a, b) = b - J(a).
\]

As before, let $Y$ be a curve, $T$ a torus, $E_T^N$ the corresponding
approximate space and $X = X(\dfS)$ be a complexity $1$ \tvar{}, where the
divisorial fan $\dfS$ is defined over the curve $Y$. Let the divisorial fan
describing $\tilde{X} := X \times E_T^N$, computed in proposition
\ref{pro:XtEc1tvr}, be denoted by $\dfS_{\tilde{X}}$. In this section, we aim to describe $\tilde{X} = X(\dfS_{\tilde{X}})$ as a \tvar{} under the action
of $T$ in terms of a semiprojective variety $\widetilde{Y_C}$ and a divisorial fan
$\dfS_{\widetilde{Y_C}}$ defined over $\widetilde{Y_C}$. The construction follows
\cite[section 5.1]{iltenvollmert:upgrading} closely. Recall that $T_E$ was the dense open torus in $E_N^T$ (see
\ref{sse:descentt}). Denote $T \times T_E$ be $\tilde{T}$. The inclusion $T
\hookrightarrow \tilde{T} = T \times T_E$ induces a surjective homomorphism
$\pi \colon \tilde{M} \longrightarrow M$ of the corresponding lattices of
characters. Let $M_E \equiv \ker \pi$ and let us define the choices of sections
and cosections of the short exact sequences of lattices using the following
diagram.
\begin{equation*}
  \begin{tikzcd}
    0 \ar[r] &
    M_E \ar[r, "\iota"'] & 
    \tilde{M} \ar[r, "\pi"'] \ar[l, bend right, "\tau"'] &
    M \ar[r] \ar[l, bend right, "\sigma^*"']& 
    0 \\
    0 &
    N_E \ar[l] \ar[r, bend left, "\tau^*"] &
    \tilde{N} \ar[l, "\rho"] \ar[r, bend left, "\sigma"] &
    N \ar[l] &
    0 \ar[l]
  \end{tikzcd}
\end{equation*}
Consider a pp-divisor $\plD= \sum \Delta_p \otimes p \in \dfS$.  Let $\omega_{\plD} \times \sigma_I^\vee \subset
\tilde{M}_{\Q}$ be the weight cone of the invariant open subset $X(\plD) \times U_{\sigma_I}$ of
$\tilde{X}$ described as a \tvar[\tilde{T}] and observe that  $
\pi(\omega_{\plD} \times \sigma_I^\vee )= M$. For $u=0 \in M$, define (compare \cite[section
5.1]{iltenvollmert:upgrading}).
\begin{equation*}
  \square_u = \tau\bigl( \pi^{-1}(u) \cap \tilde{\omega} \bigr)
  \qquad \text{and} \qquad
  \Psi_u \colon \square_u \longrightarrow \cadivQ(Y)
\end{equation*}
by $\Psi_u(u') = \plD(u' + \sigma^*(u))$. Each $\Psi_u$ is a divisorial
polyhedron. For each domain of linearity $\tilde{\omega}_i={\omega_{\plD}}_i \times \sigma_I \subset
\omega_{\plD} \times \sigma_I $ of $\plD \times \sigma_I$, choose  $u_i=0 \in \relint
\pi(\tilde{\omega}_i)=M$ in the relative interior of the image. Suppose
$\dfS_{\plD \times \sigma_I}$ be the divisorial fan corresponding to $\sum_{}^{}
\Psi_{u_i}$. Let $\tilde{Y}_{\plD \times \sigma_I}$ be the \tvar{} $X(\dfS_{\plD \times \sigma_I})$, See \cite[theroem 5.2]{iltenvollmert:upgrading}. we take $\Xi_p$ to be the coarsest polyhedral subdivision of $\rho(\Delta_p \times \sigma_I)$ containing subdivision of $\pi(\Delta)$ for any face $\Delta \prec \Delta_p \times \sigma_I $.

\begin{proposition}
  (\cite[Propositon 5.1]{iltenvollmert:upgrading}) $ \Xi_p ={\dfS_{\plD \times \sigma_I}}_p $
\end{proposition}
Before we proceed, let us recall some results from \cite{zlop95}. Given a polyhedron $\Delta$, we denote the collection of faces of a polyhedron by $L(\Delta)$.
\begin{lemma}\label{productoffaces}
    For polyhedron $\Delta$ and cone $\sigma$, $L(\Delta \times \sigma)=L(\Delta) \times L(\sigma)$, where $L(\Delta) \times L(\sigma)=\set{\Delta^\prime \times \sigma^\prime}{\Delta^\prime \prec \Delta \; \text{and} \; \sigma^\prime \prec \sigma}.$ 
\end{lemma}
\begin{proof}
    Let $F$ be a face of $\Delta \times \sigma \subset \NQ\oplus \NQ^\prime$. Consider the canonical projection maps $P : \NQ\oplus \NQ^\prime \to \NQ $ and $P^\prime: \NQ\oplus \NQ^\prime \to \NQ^\prime$. Then $F=P(F) \times P^\prime(F)$ and $P(F) \prec \Delta$, $P^\prime(F) \prec \sigma$ is evident.
\end{proof}

\begin{lemma}\label{facemapface}
    For $\Delta_p \in \dfS_p$ and $\sigma_I \in \Sigma_E$, $L(\rho(\Delta_p \times \sigma_I))= \set{\rho(\Delta)}{\Delta \prec \Delta_p \times \sigma_I) }$
\end{lemma}
\begin{proof}
   The  map $\restrict{\rho}{N \times \sigma_I} \colon N \times \sigma_I \longrightarrow N_{E}$ is injective. The projection of polyhedra $\restrict{\rho}{\Delta_p \times \sigma_I} \colon \Delta_p \times \sigma_I \longrightarrow \rho(\Delta_p \times \sigma_I) $ is bijective and $\rho$ is linear.  From \cite[Lemma 7.10]{zlop95}, inverse image of a face is a face. 
\end{proof}
\begin{lemma}\label{slicecomplete}
    If $\dfS_p$ is complete polyhedral complex then \[\set{\rho(\Delta_p \times \sigma_I)}{\Delta_p  \in \dfS_p \; \text{and} \; \sigma_I \in \Sigma_E }\] is complete polyhedral complex.
\end{lemma}
\begin{proof}
   Suppose $(a,b) \in \tilde{N}$, where $a=(a_1,a_2 \dotsc a_d) \in N$ and $b=(b_1, \dotsc b_{Nd})$ then $(a,b) \mapsto b-J(a)$. Given a element $z=(z_1 \dotsc z_N,z_{N+1}, \dotsc z_{2N} \dotsc z_{Nd}) \in N_E $ we choose $a_i$ minimum in the $N$-tuple $(z_{(i-1)N+1 , \dotsc z_{iN}})$, for $0 \leq i \leq d$. There is $b \in \sigma_I$ for some $I$, such that $(-a,b) \mapsto z$ (or $z=b-J(-a))$) and choose appropriate $\Delta_p \in \dfS_p$, such that $-a \in \Delta_p$.
\end{proof}

We are going to glue $X(\dfS_{\plD \times \sigma_I})$ for $\plD \in \dfS$ and $\sigma_I \in \Sigma_I$ to construct $\widetilde{Y_C}$. The map $\rho$ is a linear and injective on each $\Delta_p \times \sigma_I$, and faces map to faces (\ref{facemapface}). Consider set $P_{\plD \times \sigma_I}=\set{p \in Y}{\Delta_p \neq 0, \phi}$ and $E_{\plD \times \sigma_I}= \underset{p \in P_{\plD \times \sigma_I}}{\sum} p$. Then, $\dfS_{\plD \times \sigma_I}$ is divisorial fan given by the set of intersections of elements of following set 
\[ C_{\plD \times \sigma_I}= \setl{ \rho(\Delta_p \times \sigma_I) \otimes p + \phi \otimes (E_{\plD \times \sigma_I} - p)} .\]

From \cite[Section 4]{iltenvollmert:upgrading}, $\dfS_{\plD \times \sigma_I}$ is contraction free.
Consider divisorial fan $\widetilde{\dfS_C}$ set of intersection of elements of set \[C_{\dfS \times \Sigma_E}=\underset{\plD \times \sigma_I \in \dfS \times \Sigma_E}{\bigcup} C_{\plD \times \sigma_I}\]
and
\begin{notation}
    
 Let $\widetilde{Y_C}= X(\widetilde{\dfS_C}) ,$  and $\pi_C : \widetilde{X \times E^N_T} \to \widetilde{Y_C} $ be the canonical good quotient map.
 
 \end{notation}

\[
  \begin{tikzcd}
    \widetilde{X \times E^N_T} \arrow[r, "r_{E}"] \arrow[d, "\pi_C"] &
    X \times E^N_T  \\
    \widetilde{Y_C}  \arrow[d, "\pi"] \\
    Y.
  \end{tikzcd}
\]
From the discussion above, we get the following lemma.
\begin{lemma}
  The divisorial fan $\widetilde{\dfS_C}$ over $Y$ constructed above, gives us the $T_E$-variety $\widetilde{Y_C}$ which corresponds to the good quotient of   $\widetilde{X \times E^N_T}$ by the action of $T$.
\end{lemma}

Consider the following setup. Let $\Xi$ be a marked fansy divisor over $\mathbb{P}^1$, and $C$ denotes a collection of marked or contracted cones. For $\Xi$ there is a complete divisorial fan $\dfS$ such that $\Xi_p= \dfS_p$ for all $p \in \mathbb{P}^1 $. Consider complete complexity one $T$-variety, $X=X(\dfS)=X(\Xi)$.  Consider the divisorial fan \[ \dfS \times \Sigma_E = \set{\plD \times \sigma_I}{ \plD \in \dfS \; \text{and} \; \sigma_I \in \Sigma_E} \]
where $\plD = \sum \Delta_p \otimes p \in \dfS$, we defined pp-divisor,
\[ \rho(\plD \times \sigma_I)=\sum \rho(\Delta_p \times \sigma_I)  \otimes p .\]

\begin{lemma}
    The collection $\dfS_{Y_C}= \set{\rho(\plD \times \sigma_I)}{\plD \in \dfS \; \text{and} \; \sigma_I \in \Sigma_E}$ is a divisorial fan, Moreover it is complete.
\end{lemma}
\begin{proof}
   First, we are going to prove that $\rho(\plD \times \sigma_I)$ is pp-divisor on the curve $\Pl^1$. Observe that $\plD \times \sigma_I$ is pp-divisor. For the complexity one case we have an equivalent definition of pp-divisor from \cite[Section 2]{ah:affinetvar}, and we have $\text{deg}(\rho(\plD \times \sigma_I)=\rho(\text{deg}(\plD \times \sigma_I)).$ 
  The following conditions hold.
   \begin{itemize}
       \item $\rho(\plD \times \sigma_I)=\sum \rho(\Delta_p \times \sigma_I)  \otimes p$, with $\Delta_p \times \sigma_I$ is polyhedron with tail cone $\text{tail}(\plD) \times \sigma_I$ and the sum runs over all disjoint $p$'s.
       \item Since $\text{deg}(\plD \times \sigma_I)$  is proper subset of $ \text{tail}(\plD) \times \sigma_I$, $\text{deg}(\rho(\plD \times \sigma_I))$ is a proper subset of $ \rho( \text{tail}(\plD) \times \sigma_I)$.
       \item For $u \in (\text{tail}(\plD) \times \sigma_I)^\vee$, we have $ \text{eval}_u(\text{deg}(\rho(\plD \times \sigma_I)))= \text{eval}_{u\circ \rho}(\text{deg}(\plD \times \sigma_I))$ and that $\rho(\plD \times \sigma_I)(u) = (\plD \times \sigma_I)(u \circ \rho ) $ is principal.
   \end{itemize}
   To prove $\dfS_{Y_C}$ is a divisorial fan, we take two pp-divisors $\rho(\plD \times \sigma_I)$ and $\rho(\plD^\prime \times \sigma_J)$. Since  $\sigma_J \succ \sigma_J \cap \sigma_I \prec \sigma_I$, $\plD^\prime \succ \plD^\prime \cap \plD \prec \plD$ and from \ref{productoffaces}, \ref{facemapface} and \ref{slicecomplete}, we can conclude that \[ \rho(\plD \times \sigma_I) \succ \rho(\plD \times \sigma_I) \cap \rho(\plD^\prime \times \sigma_J) \prec \rho(\plD^\prime \times \sigma_J). \]
    From the criterion of completeness and the above lemma, $\dfS_{Y_C}$ is complete.
\end{proof}
Now consider the description of $\dfS_{Y_C}$ as a marked fansy divisor $\Xi_{Y_C}= \sum \dfS_{Y_C,p} \cdot p$ where, $\dfS_{Y_C,p}=\set{\rho(\Delta_p \times \sigma_I )}{\Delta_p \in \dfS_p \; \text{and} \; \sigma_I \in \Sigma_E}$ and marked cones are $C_{Y_C}=\set{\sigma \times \sigma_I}{ \; \sigma \in C}.$


\begin{lemma}
    From above lemma, we have the following diagram 

    \[ \begin{tikzcd}
      \widetilde{X(\dfS_{Y_C})} \arrow[r, "r"] \arrow[d, "\pi"] & X(\dfS_{Y_C}) \\
        \mathbb{P}^1
    \end{tikzcd},\] 
    where $\widetilde{X(\dfS_{Y_C})} =\widetilde{Y_C}.$
\end{lemma}
\begin{proof}
    Observe that $\widetilde{X(\dfS_{Y_C})} $ and $\widetilde{Y_C}$  both are contraction free with $ {\dfS_{Y_C}}_p = {\widetilde{\dfS_C}}_p$.  Here the varieties are contraction free, means that the corresponding divisorial fans are contraction free. Then, it is enough to prove that they have the same slices.  Hence, by \cite[Proposition 1.6]{ahs:gentvar}, $\widetilde{X(\dfS_{Y_C})} =\widetilde{Y_C}.$
\end{proof}
\begin{notation}
    $Y_C :=X(\dfS_{Y_C})$
\end{notation}

\begin{remark}

Note that $X \times E^N_T$ is a complexity one $T$-variety under the action of $T \times T_E$, but not complete. But $\widetilde{Y_C}$ is complete complexity one $T$-variety. 
\end{remark}

\begin{proposition}
    Consider the  morphism of pp-divisors $\plD \times \sigma_I \to \rho(\plD \times \sigma_I)$ given by the triple $(i,\rho,1)$ where $i$ is the identity morphism on $\Pl^1$, and 1 is the unit plurifunction. The map $q: X \times E^N_T \to Y_C$ induced by the morphism of pp-divisors is a geometric quotient.
\end{proposition} 
\begin{proof} The morphism $(i,\rho,1)$ gives the following ring homomorphism \[  q^\#_{\plD \times \sigma_I} :\bigoplus_{u \in \rho(\sigma \times \sigma_I)^\vee}\Gamma(\Pl^1,\mathcal{O}(\rho(\plD \times \sigma_I)(u)) \to \bigoplus_{v \in (\sigma \times \sigma_I)^\vee}\Gamma(\Pl^1,\mathcal{O}(\plD \times \sigma_I)(v)). \] where
\[\Gamma(\Pl^1,\mathcal{O}(\rho(D))(u)) \to \Gamma(\Pl^1,\mathcal{O}(D(u\circ\rho)))\] is the identity map, because $\text{min}\langle\rho(\Delta_P \times \sigma_I),u\rangle = \text{min}\langle\Delta_p \times \sigma_I,u\circ\rho\rangle$. The map $q^\#_{\plD \times \sigma_I}$ is the inclusion map, in fact \[\bigoplus_{u \in \rho(\sigma \times \sigma_I)^\vee}\Gamma(\Pl^1,\mathcal{O}(\rho(\plD \times \sigma_I)(u))=\bigoplus_{v \in (\sigma \times \sigma_I)^\vee}\Gamma(\Pl^1,\mathcal{O}(\plD \times \sigma_I)(v))^T.\] As $T$ acts freely  on $X \times E^N_T  $, the induced map $q_{\plD \times \sigma_I}: X(\plD \times \sigma_I) \to X(\rho(\plD \times \sigma_I))$ is a geometric quotient; hence the map $q$ is a geometric quotient. 
\end{proof}

Summarizing the above, we get the following lemma.

\begin{lemma}
  The $T$-variety $Y_C$ constructed above is complete and fits into a diagram.
  \[
  \begin{tikzcd}
    \widetilde{X \times E^N_T} \arrow[r, "r_{E}"] \arrow[d, "\pi_C"] &
    X \times E^N_T \arrow[d, "q"] \\
   \widetilde{Y_C} \arrow[r] \arrow[d, "\pi"] &
    Y_C\\
    \mathbb{P}^1,
  \end{tikzcd}
  \]
  where $r_{E}$ is a $T$-equivariant birational proper morphism, and the maps $q$ and $\pi_C$ are geometric quotients.
\end{lemma}

\section{Torus equivariant Chow groups of \tvars{}}
\label{sec:teqchtvr}

Let us continue the setup above, with $\text{tail}(\dfS)= \Sigma$. Denote the set of points $p \in \mathbb{P}^1$, such that slice $\dfS_p \neq \Sigma$, by $P$. We shall ensure that the size of $P$ is at least 2 by appending extra points, if necessary. For a non-negative integer $k \leq d+1$, we obtain the $k^{th}$ Chow group of a $T$-variety of complexity one, consider the following sets

\begin{itemize}
    \item $R_k =$  cones of dimension $d+1-k$ corresponding to subvarieties not contracted by $r$.
   
    \item $V_k=$  faces of dimension $d-k$ of polyhedral subdivision corresponding to the fiber of points in $P$.
    \item $T_k =$ cones of dimension $d-k$ corresponding to subvarieties contracted by $r$.
\end{itemize}
\begin{theorem}(\cite[Theorem 4.1]{Ndland2021})
   Given a complete, complexity one rational $T$-variety $X$, for $0 \leq k \leq \text{dim}(X)$, one has an exact sequence,
  \[ \begin{tikzcd}
      K \arrow[r] & \mathbb{Z}^{V_k} \oplus \mathbb{Z}^{R_k} \oplus \mathbb{Z}^{V_k} \arrow[r] & A_{k}(X) \arrow[r] & 0
   \end{tikzcd} \]      
    where arrows and lattice $K$ are defined in \cite[Theorem 4.1]{Ndland2021}.
\end{theorem}
  Note that in our case, $Y_C$ is a $T$-variety defined by a pp-divisor over $\mathbb{P}^1$. Thus it contains an open subset, which is a product of an open subset of $\mathbb{P}^1$ with a torus, and hence $Y_C$ is rational.
 Using the above result, we will give a presentation of the equivariant Chow group of $X(\dfS)$ or, equivalently Chow group of $Y_C$. For non negative integer $k \leq Nd+1$, we define the following sets:
 \begin{itemize}
 \item $R^\prime_k =$ \# faces of dimension $k$ in $\Sigma_E$.
 \item $r_k =$  cones of dimension $Nd+1-k$ corresponding to subvarities not contracted by $r$.
   
    \item $v_k=$  faces of dimension $Nd-k$ of polyhedral subdivision corresponding to the fiber of points in $P$.
    \item $t_k =$ cones of dimension $Nd-k$ corresponding to subvarities contracted by $r$.
    \end{itemize}

The cardinalities of sets $r_k, v_k, t_k$ are given by following numbers:
    
    \begin{itemize}

      \item $|r_k|= R^\prime_{Nd-d}.|R_k| +R^\prime_{Nd-d-1}.|R_{k-1}| \dotsc R^\prime_{Nd-d-(k-1)}.|R_{1}|+ R^\prime_{Nd-d-k}.|R_{0}|$.

    \item $|v_k| = R^\prime_{Nd-d}.|V_{k}|+R^\prime_{Nd-d-1}.|V_{k-1}| \dotsc R^\prime_{Nd-d-(k-1)}.|V_{1}|+R^\prime_{Nd-d-k}.|V_{0}|$.

    \item $|t_k|= R^\prime_{Nd-d}.|T_{k}|+  R^\prime_{Nd-d-1}.|T_{k-1}|+ \dotsc  R^\prime_{Nd-d-(k-1)}.|T_{1}|+  R^\prime_{Nd-d-k}.|T_{0}|.$
 \end{itemize}

\section{$|r_k| + |v_k| + |t_k|$ is constant}
\label{sec:rkvktkcn}

From \cite[Proposition 6.3]{Ndland2021} and from above calculation we can state following result.

For
\begin{equation} \label{EqnforX}
S_i=|R_i| + |V_i| + |T_i| = 
\left\{
    \begin{array}{lr}
        \#\Sigma(d-i+1) + 2\#\Sigma(d-i), & \text{if } i < d;\\
         \#\Sigma(1) + \#P, & \text{if } i= d;\\
         0,& \text{if } i>d;
    \end{array}
\right.
\end{equation}
and 
\begin{equation}
    S_i^\prime=
 \left\{   \begin{array}{cc}
      
   R_{Nd-d-i}^\prime= \Sigma_T(Nd-d-i), & \text{if } i \leq Nd-d;\\
    0, & \text{if } i > Nd-d.
\end{array}
\right.\label{EqnforBG}
\end{equation}

\begin{proposition} \label{pro:rvksumcn}
    For any rank two toric vector bundle $\mathscr{E}$ on a smooth toric variety $X_{\Sigma} $ we have $X=P(\mathscr{E})$ and $X_\mathscr{E}=X \times \mathscr{E}^N_T/T$,The numbers $|r_k|$, $|v_k|$, and $|t_k|$ are associated with $X_\mathscr{E}$ then \[|r_k| + |v_k| + |t_k| = \sum_{i=0}^{i=k}S^\prime_{i} S_{k-i}.\]
    \begin{proof}
        Follows from \ref{EqnforX} and \ref{EqnforBG} and \cite[Proposition 6.3]{Ndland2021}
    \end{proof}
\end{proposition}

\begin{example} \label{exa:rkvktkeg}
Consider the example \cite[Example 5.3]{Ndland2021}, for vector bundles $\mathscr{E}$ and $\mathscr{F}$ from \cite[Remark 6.5]{Ndland2021} $S_j$ are independent of $\mathscr{E}$ or $\mathscr{F}$. We also fix a large enough value of $N$, then we have an integer $|r_k| + |v_k| + |t_k| $
is independent of $\mathscr{E}$ or $\mathscr{F}$. Note that corresponding to $\mathscr{E}$, we have space  $\Pl(\mathscr{E}) \times \BGtorus{d}/T$, and for $\mathscr{F}$, we have $\Pl(\mathscr{F}) \times \BGtorus{d}/T$.
Consider $\mathbb{P}^2$ with the torus action induced by the fan given by the rays $\rho_1=(1,0)$,$\rho_2=(0,1)$, and $\rho_0=(-1,-1)$ and denote the associated divisors by $D_i$, for $i=1,2,0$, respectively. Consider rank two vector bundles given by $\mathscr{E}= D_1 \oplus 0$ and $\mathscr{F}= (D_1 +D_2) \oplus D_0$. Then we have that $\mathbb{P}(\mathscr{E})$ and $\mathbb{P}(\mathscr{F})$ are complexity one $T$-varieties of dimension 3. 

 The following table demonstrates the values of $|r_1|$, $|v_1|$, $|t_1|$, $|r_2|$, $|v_2|$ and $|t_2|$ in this case.
\begin{equation*}
 \begin{tabular}{|c||c|c|c|c|c|c|}
\hline
& $|r_2|$ & $|v_2|$ & $|t_2|$ & $|r_1|$ & $|v_1|$ & $|t_1|$  \\
\hline
\hline
     $(\mathbb{P}(\mathscr{E})\times E^3_2 ) /T $ & 213 & 48 & 36 & 135 & 45 & 9  \\
 \hline
     $(\mathbb{P}(\mathscr{F}) \times E^3_2) /T $ & 132 & 165 & 0 & 54 & 135 & 0  \\    
 \hline 
 
\end{tabular}   
\end{equation*}
\end{example}
 
\bibliographystyle{amsalpha}
\bibliography{references}
\end{document}